\documentclass{amsart}[11pt]
\usepackage{amsmath,amsthm,amsfonts,amssymb,amscd,enumerate,verbatim}
\usepackage{varwidth}
\usepackage{graphicx}
\usepackage[hidelinks=true]{hyperref}
\usepackage{colonequals}

\usepackage{stmaryrd}

\usepackage[all]{xy}\SelectTips{eu}{}

\usepackage{color}

\newcommand{\ov}{\overline}
\newcommand{\wt}{\widetilde}

\newcommand{\col}{\colon}

\newcommand{\fm}{{\mathfrak m}}
\newcommand{\ann}{\operatorname{ann}}

\newcommand{\fn}{{\mathfrak n}}
\newcommand{\p}{{\mathfrak p}}
\newcommand{\q}{{\mathfrak q}}
\newcommand{\fq}{{\mathfrak q}}
\newcommand{\Podd}{\operatorname{P}_{\operatorname{odd}}^{M,N}}
\newcommand{\Peven}{\operatorname{P}_{\operatorname{even}}^{M,N}}

\newcommand{\cx}{\operatorname{cx}}

\newcommand{\codim}{\operatorname{codim}}
\newcommand{\depth}{\operatorname{depth}}

\newcommand{\Tor}{\operatorname{Tor}}
\newcommand{\fext}[1]{\operatorname{f}_{#1}}
\newcommand{\Ext}{\operatorname{Ext}}

\newcommand{\Hom}{\operatorname{Hom}}
\newcommand{\Po}{\operatorname{P}}

\newcommand{\syz}[3][R]{{\Omega^{#1}_{#2}#3}}

\newcommand{\V}{\operatorname{V}}
\newcommand{\Proj}[1]{\operatorname{Proj}(#1)}
\newcommand{\Supp}{\operatorname{Supp}}
\newcommand{\Coker}{\operatorname{Coker}}

\numberwithin{equation}{section}

\theoremstyle{plain}
\newtheorem{theorem}{Theorem}[section]

\newtheorem*{Theorem1}{Theorem 1}
\newtheorem*{Theorem2}{Theorem 2}

\newtheorem*{Main Theorem}{Main Theorem}
\newtheorem*{Corollary}{Corollary}

\newtheorem{proposition}[theorem]{Proposition}
\newtheorem{lemma}[theorem]{Lemma}
\newtheorem{corollary}[theorem]{Corollary}

{\Alph{theorem}}
{\Alph{theorem}}

\theoremstyle{definition}

\newtheorem{chunk}[theorem]{}

\newtheorem{remark}[theorem]{Remark}
\newtheorem{definition}[theorem]{Definition}
\newtheorem{example}[theorem]{Example}

  \newcounter{numlist} \newenvironment{numlist}{\begin{list}{\upshape
      (\arabic{numlist})}%
    {\usecounter{numlist}%
      \setlength{\leftmargin}{3em}%
      \setlength{\labelwidth}{3em}%
      \setlength{\rightmargin}{0em}%
      \setlength{\partopsep}{0em}%
      \setlength{\topsep}{0ex}%
      \setlength{\parsep}{0em}%
      \setlength{\itemsep}{0ex}}}%
  {\end{list}}%

\theoremstyle{remark}
\newenvironment{bfchunk}{\begin{chunk}\textit}{\end{chunk}}

\numberwithin{equation}{theorem}

\newtheorem*{Case1}{Case 1}
\newtheorem*{Case2}{Case 2}

\begin{document}
\title[Asymptotic behavior of Ext]{Asymptotic behavior of Ext for pairs of modules of large complexity over graded complete intersections}

\author{David A. Jorgensen}
\address[D. A. Jorgensen]{Department of Mathematics, University of Texas at Arlington, 411 S. Nedderman Drive, Pickard Hall 429, Arlington, TX 76019, USA}
\email{djorgens@uta.edu}
\urladdr{http://www.uta.edu/faculty/djorgens/}
\author{Liana M. \c Sega}
\address[L. M. \c Sega]{Department of Mathematics and Statistics, University of Missouri - Kansas City, 206 Haag Hall, 5100 Rockhill Road, Kansas City, MO 64110-2499, USA}
\email{segal@umkc.edu}
\urladdr{http://s.web.umkc.edu/segal/}
\author{Peder Thompson}
\address[P. Thompson]{Institutt for matematiske fag, NTNU, N-7491 Trondheim, Norway.
Current address: Department of Mathematics, Niagara University, NY 14109, USA}
\email{thompson@niagara.edu}
\urladdr{http://pthompson.nupurple.net/}

\date{August 10, 2022}
\thanks{This work was supported in
part by a Simons Foundation grant (\#354594, Liana \c Sega)}

\begin{abstract}
Let $M$ and $N$ be finitely generated graded modules over a graded complete intersection $R$ such that $\Ext_R^i(M,N)$ has finite length for all $i\gg 0$. We show that the even and odd Hilbert polynomials, which give the lengths of $\Ext^i_R(M,N)$ for all large even $i$ and all large odd $i$, have the same degree and leading coefficient whenever the highest degree of these polynomials is at least the dimension of $M$ or $N$.  Refinements of this result are given when $R$ is regular in small codimensions. 
\end{abstract}

\maketitle

\section*{introduction}
\noindent
Let $Q$ be a regular local ring with maximal ideal $\fn$ and $\underline{f}=f_1, \dots, f_c$ be a $Q$-regular sequence contained in $\fn^2$. Set $R=Q/(\underline{f})$ and let $M$ and $N$ be finitely generated $R$-modules.  In this case $R$ is called a \emph{complete intersection} and Gulliksen \cite{Gu} has shown that $\Ext_R(M,N)$ has the structure of a finitely generated graded module over the polynomial ring $R[\chi_1,\dots,\chi_c]$, where $\deg\chi_i=2$ for $1\le i\le c$. Consequently, $\Ext_R(M,N)$ decomposes as a direct sum of two graded submodules, one consisting of the $\Ext_R^i(M,N)$ for $i$ even, and the other of the $\Ext_R^i(M,N)$ for $i$ odd.  

Further assume that $\Ext_R^i(M,N)$ has finite length for all $i\gg 0$. The tails of the even and odd Ext modules are thus finitely generated over the polynomial ring $R/\fn^n[\chi_1,\dots,\chi_c]$ for some integer $n\geq 1$. Standard commutative algebra then tells us that there exist polynomials $\Peven(t)$ and $\Podd(t)$ with rational coefficients, which we call the \emph{even} and \emph{odd Hilbert polynomials} of $\Ext_R(M,N)$, such that $\Peven(i)=\lambda(\Ext_R^i(M,N))$ for all even $i\gg 0$ and $\Podd(i)=\lambda(\Ext_R^i(M,N))$ for all odd $i\gg 0$, where $\lambda(-)$ denotes length.

A natural question to ask is whether $\Peven$ and $\Podd$  have the same degree and, if so, the same leading coefficient.  There are easy and well-known examples showing that this is not always the case. For example, consider the ring $R=k\llbracket x_1,x_2\rrbracket /(x_1x_2)$ and the $R$-modules 
$M=R/(x_1)$ and $N=R/(x_2)$. With these choices, $\Podd=1$ and $\Peven=0$. However, for $N=k$, where $k=Q/\fn$ (so that the integers $\lambda(\Ext_R^i(M,N))$ are the Betti numbers of $M$), the polynomials $\Peven$ and $\Podd$ do have the same degree and leading coefficient, for example by the proof of Avramov's \cite[Theorem 9.2.1(1)]{Av}; see also Celikbas and Dao \cite[Proposition 3.2]{CD}.

The above statements translate---as is standard---to the graded setting, and in this paper we establish equality of degree and leading coefficient of the Hilbert polynomials $\Peven$ and $\Podd$ for a broad contingent of pairs of finitely generated graded modules over a \emph{graded} complete intersection. 
The maximum degree of $\Peven$ and $\Podd$ is one fewer than the \emph{complexity} of the pair $(M,N)$, denoted $\cx_R(M,N)$, and we formulate our results in terms of this established invariant; see Section \ref{Laurent-coefficients}. The complexity $\cx_R(M,N)$ is also the dimension of the support variety of $M$ and $N$; see Avramov and Buchweitz \cite{AB}. The following is part of Theorem \ref{dimensions} below.

\begin{Theorem1}
Let $R$ be a graded complete intersection and let $M$ and $N$ be finitely generated graded $R$-modules with $\lambda(\Ext_R^i(M,N))<\infty$ for all $i\gg 0$. If 
 $$\cx_R(M,N)> \max\{\dim R/(\ann M+\ann N), \dim M+\dim N-\dim R\}$$
 then $\Peven$ and $\Podd$ have the same degree and leading coefficient.
\end{Theorem1}

\noindent
In particular, the conclusion holds when $\cx_R(M,N)> \min\{\dim M, \dim N\}$.  It also holds for possibly smaller values of $\cx_R(M,N)$, provided $R_\p$ is regular for all homogeneous prime ideals $\p$ of small codimension, see Theorem \ref{dimensions}(3) and Theorem \ref{codimension} below.

Before describing the paper further, we comment that its results complement those of several recent articles on certain homological invariants for pairs of modules over hypersurfaces and, more generally, over complete intersections.  The Herbrand difference introduced by Buchweitz \cite{Bu} for hypersurfaces, and generalized to complete intersections of higher codimension in \cite{CD},  is one such invariant; see \eqref{Herbrand} below for the definition of the $j$th Herbrand difference $h_j^R(M,N)$. 

For the remainder of the introduction, unless otherwise specified, $R$ is a graded complete intersection of codimension $c$, and $M$ and $N$ are finitely generated graded $R$-modules with $\lambda(\Ext_R^i(M,N))<\infty$ for all $i\gg0$. Set $r=\cx_R(M,N)$.
As is pointed out by Moore, Piepmeyer, Spiroff, and Walker 
\cite{MPSW2}, when $r>0$ one has that $\Peven$ and $\Podd$ have the same degree and leading coefficient if and only if $h_{r}^R(M,N)=0$.  
Several recent papers establish the vanishing of $h_{c}^R(M,N)$ under certain conditions on the ring and (or) on the modules, see \cite{BMTW,D1,D2,DK,MPSW,MPSW2,W1,W}. 
Most of these results are formulated in terms of vanishing of higher codimension versions of Hochster's theta pairing, such as Dao's eta-invariant $\eta_r^R(M,N)$ in \cite{D1}; these invariants are defined similarly, only using $\Tor$ instead of $\Ext$. When $R$ is an isolated singularity, the vanishing of $\eta_{c}^R(-,-)$ for all pairs of finitely generated $R$-modules is equivalent to the vanishing of $h_{c}^R(-,-)$ for all pairs of finitely generated $R$-modules, due to a result of Dao \cite[Lemma 5.9]{W1}. 

The origin of these invariants goes back to Serre and his work on intersection multiplicities \cite{Serre}. Hochster defined his theta invariant in order to extend Serre's work to hypersurfaces.  He showed in \cite{Hoc81} that his theta invariant for a pair of finitely generated modules $M$ and $N$ satisfying $\lambda(M\otimes_R N)<\infty$ over a so-called admissible hypersurface vanishes if and only if $\dim M+\dim N\le \dim R$. Theorem 1 above provides a cohomological extension of Hochster's result to graded complete intersections of arbitrary codimension. 

\begin{Corollary} Let $R$ be a graded complete intersection and let $M$ and $N$ be finitely generated graded $R$-modules with $\cx_R(M,N)>\dim R/(\ann M+\ann N)$ and such that $\lambda(\Ext_R^i(M,N))<\infty$ for all $i\gg 0$. Then $\Peven$ and $\Podd$ have the same degree and leading coefficient if and only if $\dim M+\dim N-\dim R<\cx_R(M,N)$.
\end{Corollary}

This is Corollary \ref{Hochster_corollary} below. If $\lambda(M\otimes_RN)<\infty$, then $\dim R/(\ann M+\ann N)=0$ and $\lambda(\Ext_R^i(M,N))<\infty$ for all $i\geq 0$; see Remark \ref{finitelength_rmk}. Thus the hypothesis of Hochster's result implies the hypothesis of the corollary when $\cx_R(M,N)=1$, and in this case the conclusion of the corollary gives the conclusion of a cohomological version of Hochster's result; Remark \ref{Hochster_remark} shows that a cohomological version of Hochster's result also holds in the case $\cx_R(M,N)=0$.  

From \cite[Theorem 5.6(1),(2)]{AB} and \cite[Proposition 2.2]{CD} we have the inequality $r\le c$, where $c$ is the codimension of $R$ and 
$r=\cx_R(M,N)$ as above. Furthermore, pairs of finitely generated  graded $R$-modules ($M$, $N$), neither of which is $k$, can be chosen so that $r$ takes any given value between $1$ and $c$ (or $-\infty$), see Proposition \ref{Dave}.  Since the results in the recent literature only establish the vanishing of $h_{r}^R(M,N)$ when $r=c$,  such results have no bearing on the question of whether the degree and leading coefficients of $\Peven$ and $\Podd$ are equal when $r<c$.  One of the major ingredients used in the proofs of the  existing  results is that, when $R$ is an isolated singularity, the pairings $h_c^R(-,-)$ and $\eta_c^R(-,-)$ are biadditive on short exact sequence of finitely generated $R$-modules, and this allows for work in Grothendieck groups. On the other hand, if $1\le r<c$, the pairing $h_{r}^R(-,-)$ is only biadditive on certain short exact sequences, see \cite[Theorem 3.4(2)]{CD}. 

The main ingredient in our work consists of uncovering a connection between $\Peven$ and $\Podd$ having the same degree and leading coefficient and certain invariants in work of Avramov, Buchweitz, and Sally \cite{ABS}. Let $H(M,t)$ be the Hilbert series of $M$ and for each $j\ge 0$ let $\rho_R^j(M,N)(t)$ denote the following rational function: 
$$
\rho_R^j(M,N)(t)=\sum_{i=0}^j(-1)^i H({\Ext^i_R(M,N)},t)-\frac{H(M,t^{-1})H(N,t)}{H(R,t^{-1})}\,.
$$
Further, let $o(\rho_R^j(M,N))$ denote the order of the Laurent series expansion around $t=1$ of 
$\rho_R^j(M,N)(t)$, see Section \ref{Laurent-coefficients} for details.  The following is Theorem \ref{criterion} below.

\begin{Theorem2} Let $R$ be a graded complete intersection and let $M$ and $N$ be finitely generated graded $R$-modules. Set $\ell=\inf\{i\mid\lambda(\Ext^j_R(M,N))<\infty\text{ for all }j\ge i\}$ and $r=\cx_R(M,N)$.  If $\ell<\infty$ and $r\ge 1$, then the following hold: 
\begin{numlist}
\item  $o(\rho_R^\ell(M,N))\ge -r$, and
\item The polynomials $\Peven$ and $\Podd$ have the same degree and leading coefficient if and only if  
$o(\rho_R^{\ell}(M,N))>-r$. 
\end{numlist}
\end{Theorem2}
\noindent
The inequality in (1) of Theorem 2 can be interpreted in terms of relations between the Laurent coefficients of the Laurent series expansions around $t=1$ of $H(\Ext^i_R(M,N),t)$ for $i=1, \dots, {\ell }$,  and those of $H(M,t)$, $H(N,t)$, and $H(R,t)$, as in  \cite[Theorem 7]{ABS}.   Part (2) of Theorem 2 provides the main ingredient in proving our results about the odd and even Hilbert polynomials of $\Ext_R(M,N)$ and it  leads directly to the proof of  Theorem 1.  In addition, the proofs of Theorem \ref{dimensions}(3) and Theorem \ref{codimension} use several results from \cite{ABS}.

The structure of the paper is as follows. Section \ref{Laurent-coefficients} introduces notation and records properties of Laurent coefficients,  and Section \ref{regularity_conditions} further investigates such properties under the additional condition that $R_{\p}$ is regular for homogeneous prime ideals $\p$ of certain codimensions.
Section \ref{reducing-complexity} presents a construction, which is possible when the base field is infinite, of a module $K$ such that the maximum degree of the even and odd Hilbert polynomials of the pair $(K,N)$ is one less than that of the pair $(M,N)$, and the even and odd Hilbert polynomials of the pair $(M,N)$ have the same degree and leading coefficient if and only if the same is true for the pair $(K,N)$; this is a standard construction of superficial elements that exploits the finite generation of $\Ext_R(M,N)$ over the ring $R[\chi_1, \dots, \chi_c]$ of cohomology operators. Section \ref{criterion-vanishing} gives a proof of Theorem 2, using computations with Hilbert series, Laurent expansions, and a repeated application of the construction from Section \ref{reducing-complexity}, which then also yields a proof of Theorem 1. Section \ref{arbitrary_complexity} contains examples of pairs of modules having arbitrary complexity, and shows that the regularity hypotheses in Theorem \ref{codimension} are necessary for pairs of modules of smaller complexity.

\section{Notation and properties of Laurent series}
\label{Laurent-coefficients}
\noindent
 
\begin{bfchunk}{Setting.}
\label{setting}
Let $Q=k[x_1, \dots, x_\nu]$ be a polynomial ring over a field $k$, with variables $x_i$ of positive degree,  and $R=Q/I$ be  a graded ring, with $I$ a homogeneous ideal such that $I\subseteq (x_1, \dots, x_\nu)^2$. Let $M$ and $N$ be finitely generated graded $R$-modules. 
\end{bfchunk}

 We denote length by $\lambda(-)$ and we set\\
\begin{align*}
\fext{R}(M,N)&=\inf\{n\in \mathbb{Z}\mid \lambda_R(\Ext^i_R(M,N))<\infty \text{ for all $i> n$}\}\quad\text{and}\\
\beta_i^R(M,N)&=\lambda_R(\Ext^i_R(M,N))\,,\quad\text{for each } i> \fext{R}(M,N)\,.
\end{align*}

We say $R$ is a \emph{graded complete intersection} (of codimension $c$) if $R$ is as in \ref{setting}, with $I=(f_1, \dots, f_c)$ for a homogeneous $Q$-regular sequence $f_1, \dots, f_c$. 
Definitions and results for local complete intersections translate to the graded setting in a standard manner. In particular, if  $R$ is a graded complete intersection with $\fext{R}(M,N)<\infty$ then there exist polynomials $\Peven$ and 
$\Podd$ with rational coefficients such that                                                        
$$
\Peven(2i)=\beta_{2i}^R(M,N)\quad\text{and}\quad \Podd(2i+1)=\beta_{2i+1}^R(M,N)\quad\text{
for all $i\gg 0$.}
$$
See \cite[proof of Theorem 9.2.1(1)]{Av} and \cite[Proposition 3.2]{CD}.

Assume $\fext{R}(M,N)<\infty$. The {\it complexity} of a pair $(M,N)$ of finitely generated graded $R$-modules is defined as
\begin{equation*}
\cx_R(M,N)=\inf\left\{b\in \mathbb N \;\middle|\; 
    \begin{gathered}
     \beta_n^R(M,N)\le an^{b-1}\text{ for some}\\
     \text{real number $a$ and for all $n\gg 0$}
     \end{gathered}\;\right\}\,.
\end{equation*}
\noindent
This definition of complexity agrees, by \cite[Proposition 2.2]{CD}, with the more general one given by Avramov and Buchweitz \cite{AB}. 
As noted in the introduction, 
$$\cx_R(M,N)=1+\max\left\{\operatorname{deg}\Peven,\operatorname{deg}\Podd\right\}.$$ 
Indeed, if $b=\cx_R(M,N)$, then from the inequalities
\begin{align*}
\Peven(2i)&=\beta_{2i}^R(M,N)\le a(2i)^{b-1}\;, \text{ and}\\
\Podd(2i+1)&=\beta_{2i+1}^R(M,N)\le a(2i+1)^{b-1}
\end{align*}
for all $i\gg 0$, it follows that $b-1$ must be the larger of $\operatorname{deg}\Peven$
and $\operatorname{deg}\Podd$.

We next define an invariant that measures when $\Peven$ and $\Podd$ have the same degree and leading coefficient.
\begin{definition}\label{h_dfn}
Assume $R$ is a graded complete intersection. Let $M$ and $N$ be finitely generated graded $R$-modules such that $\fext{R}(M,N)<\infty$ and set $r=\cx_R(M,N)$. Since $r<\infty$, when $r>0$ we may write
\begin{align*}
\Peven(t)&=a_{r-1}t^{r-1}+\dots +a_1t+a_0\,,\\
\Podd(t)&=b_{r-1}t^{r-1}+\dots +b_1t+b_0\,,
\end{align*}
where at least one of $a_{r-1}$ and $b_{r-1}$ is nonzero. If $r>0$, we define 
$$
h_R(M,N)\colonequals a_{r-1}-b_{r-1}\,. 
$$
If $r=0$, we set $h_R(M,N)=0$. 

\begin{remark}\label{h_rmk}
The invariant $h_R(M,N)$ equals the classical Herbrand difference from \cite{Bu} when $c=r=1$. 
\end{remark}

For any non-negative integer $j$, the $j$th \emph{(generalized) Herbrand difference} of $M$ and $N$ is defined in \cite{CD} as
\begin{equation}
\label{Herbrand}
h_j^R(M,N)=\lim_{n\to \infty}\frac{\sum_{i=\fext{R}(M,N)+1}^n(-1)^i\beta_i^R(M,N)}{n^j}\,.
\end{equation}
\end{definition}

\begin{proposition}
Assume $R$ is a graded complete intersection. Let $M$ and $N$ be finitely generated graded $R$-modules with $\fext{R}(M,N)<\infty$ and set $r=\cx_R(M,N)$. One has $h_R(M,N)=2rh_r^R(M,N)$, and consequently if $r>0$, then the invariants $h_R(M,N)$ and $h_r^R(M,N)$ vanish simultaneously.
\end{proposition}
\begin{proof}
If $r=0$, there is nothing to show, so we may assume $r>0$.
By \cite[Proposition 3.2(2,3)]{CD}, there are rational numbers $u$ and $v$, a nonnegative integer $s$, and polynomials $q_{0}(t),q_{1}(t)\in \mathbb{Q}[t]$  of degree at most $r-2$ such that for $i\gg0$,
$$\beta_i^R(M,N)=ui^{r-1}+(-1)^ivi^{s-1}+q_{i \bmod 2}(i),$$
where $r\geq s$. Choose $h>\fext{R}(M,N)$ such that $\beta_i^R(M,N)$ is given by the previous display for $i\geq h$. It follows, for example from \cite{Kim05} and Faulhaber's formula, that $\sum_{i=h}^n(-1)^iui^{r-1}$ and $\sum_{i=h}^nq_{i \bmod 2}(i)$ are polynomials in $n$ of degree at most $r-1$, hence 
\begin{equation}
\label{lim_compare}
h_r^R(M,N)=\lim_{n\to\infty}\frac{\sum_{i=h}^nvi^{s-1}}{n^r}=\frac{v}{s}\lim_{n\to\infty}n^{s-r}.
\end{equation}
If $s<r$, then $a_{r-1}=u=b_{r-1}$, and it is clear that $h_r^R(M,N)=0=h_R(M,N)$. On the other hand, if $s=r$, then 
$$h_R(M,N)=a_{r-1}-b_{r-1}=(u+v)-(u-v)=2v.$$
Thus by \eqref{lim_compare}, one obtains $h_R(M,N)=2rh_r^R(M,N)$.
\end{proof}

\begin{remark}
The condition $\fext{R}(M,N)<\infty$ holds for any pair of finitely generated graded $R$-modules if $R$ is an isolated singularity. Unlike most results in the recent literature, we do not require that $R$ is an isolated singularity.
\end{remark}

\begin{remark}\label{finitelength_rmk}
It is standard that if $\lambda(M\otimes_R N)<\infty$, then $\lambda(\Ext_R^i(M,N))<\infty$ for all $i\geq 0$. For convenience, we give a short proof. First recall, for example from \cite[Section 6]{Mat89}, that a nonzero $R$-module has finite length if and only if it is finitely generated with support consisting only of maximal ideals. For every non-maximal prime $\p$, the condition $\lambda(M\otimes_R N)<\infty$ implies that $M_\p\otimes_{R_\p}N_\p=0$, hence $M_\p=0$ or $N_\p=0$. Thus $\Ext_R^i(M,N)_\p\cong \Ext_{R_\p}^i(M_\p,N_\p)=0$, and so $\lambda(\Ext_R^i(M,N))<\infty$ for $i\geq0$.\end{remark}

 For the remainder of this section we do not need to assume that $R$ is  a complete intersection. Next we introduce invariants defined in terms of the coefficients of the Laurent series expansions around $t=1$ of certain rational functions involving the Hilbert series of the modules $M$ and $N$ and their $\Ext$ modules, and discuss their properties. We begin by discussing Laurent series expansions of rational functions in general. 

Let $\varphi(t)\in \mathbb{R}(t)$ be a rational function.

\begin{bfchunk}{Laurent series.} Let 
$${\displaystyle \sum_{n=-\infty}^{\infty} a_n(t-1)^n}$$ be the Laurent series expansion of $\varphi(t)$ around $t=1$. The {\it order} of this Laurent series is denoted $o(\varphi)$ and is defined as  
$$
o(\varphi)=\inf\{n\in \mathbb Z\mid a_n\ne 0\}\,.
$$ 
When $\varphi=0$, the order $o(\varphi)$ is defined to be $\infty$. 
For each integer $n$ we set
\[
g^n(\varphi)=(-1)^na_{-n}\,.
\]
With this notation, we  have 
\[
o(\varphi)=-\sup\{n\in \mathbb Z\mid g^n(\varphi)\ne 0\}
\]
 and the Laurent series of $\varphi(t)$ can be rewritten as 
$$
{\displaystyle\sum_{n\le -o(\varphi)} \frac{g^n(\varphi)}{(1-t)^n}}\,.
$$
\end{bfchunk}

\begin{bfchunk}{Poles and orders.}
\label{varphi-general}
Since $\varphi$ is a rational function, we have $o(\varphi)>-\infty$.
We say that $\varphi$ has a pole at $t=1$ if $1$ is a zero of the denominator of $\varphi$ when $\varphi$ is written in reduced form. If $\varphi$ has a pole at $t=1$, we say that the pole has order $m$ if 
$$\varphi(t)=\displaystyle{\frac{Q(t)}{(1-t)^m}}$$
for some rational function $Q(t)$ with no pole at $t=1$ and $Q(1)\ne 0$. 

If  $\varphi$ does not have a pole at $t=1$ then $o(\varphi)\ge 0$. Otherwise, we see that $\varphi$  has a pole of order $m>0$ at $t=1$ if and only if  $o(\varphi)=-m$. In this case, we  have $g^m(\varphi)=Q(1)$ and $g^a(\varphi)=0$ for all $a>m$, where $Q(t)$ is as above.  

Note that $o(\varphi)=0$ if and only if $\varphi$ does not have a pole at $t=1$ and $\varphi(1)\ne 0$. Also, $o(\varphi)\ge 1$ if and only if $\varphi$ does not have a pole at $t=1$ and $\varphi(1)= 0$. 
\end{bfchunk}

\begin{bfchunk}{Sums and products.}
\label{product}
Let $\psi$ be another rational function. Set $m=-o(\varphi)$ and $n=-o(\psi)$.
When adding Laurent series, one simply adds the coefficients, and we have
\begin{equation}
\label{add-orders}
o(\varphi+\psi)\ge \min\{o(\varphi), o(\psi)\}\, ,
\end{equation}
with strict inequality if and only if  $m=n$ and $g^m(\varphi)+g^m(\psi)= 0$. 

In general, Laurent series cannot be multiplied because computing the coefficients of a product leads to infinite sums. However, the problem is avoided when the series have order different than $-\infty$, as is the case for rational functions:
\begin{equation}
\label{order-sum}
o(\varphi\psi)=o(\varphi)+o(\psi)\quad\text{and}\quad o\left(\frac{\varphi}{\psi}\right)=o(\varphi)-o(\psi) \quad\text{when $\psi\ne 0$.}
\end{equation}
The coefficients of the Laurent series of  $\varphi\psi$ can be computed using the formula
$$
g^a(\varphi\psi)=\sum_{i=a-n}^{m}g^i(\varphi)g^{a-i}(\psi)\,.
$$
In particular, we have 
\begin{align*}
g^{m+n}(\varphi\psi)&=g^m(\varphi)g^n(\psi)\,,\text{ and}\\
g^{m-n}(\frac{\varphi}{\psi})&=\frac{g^m(\varphi)}{g^n(\psi)}\quad \text{when $\psi\ne 0$}\,.
\end{align*}
\end{bfchunk}

\begin{bfchunk}{Hilbert series.}
\label{Hilbert series}
The Hilbert series of a graded $R$-module $M=\oplus_{i\ge b} M_i$,  where $b\in \mathbb{Z}$, is
$$
H(M, t)=\sum_{i\ge b}\dim_{k}(M_i)t^i\in \mathbb Z\llbracket t\rrbracket[t^{-1}]\,.
$$

Let $a$ be an integer and let $M(a)$ denote the graded $R$-module with $M(a)_i=M_{a+i}$.  
Note that 
$$H(M(a),t)=t^{-a}H(M,t)\,.$$ 
The following is a standard result in dimension theory, for example see \cite[Proposition 4.4.1]{BH98}:
If $\dim M=m$, then there exist positive integers $s_1, \dots, s_m$  and $h(t)\in \mathbb Z[t,t^{-1}]$ where $h(1)\ne 0$ such that 
\begin{equation*}
H(M,t)=\frac{h(t)}{\prod_{i=1}^m (1-t^{s_i})}\,.
\end{equation*}
In particular, the Hilbert series $H(M, t)$ is a rational function that can be further written as $H(M,t)={\displaystyle \frac{Q_M(t)}{(1-t)^m}}$, where $Q_M(t)$ is a rational function with no pole at $t=1$ and $Q_M(1)\ne 0$.  To simplify notation, we set for each integer $i$
$$
g^i(M)=g^i(H(M,t))\quad\text{and}\quad o(M)=o(H(M,t))\,.
$$
We have: 
\begin{equation}
\label{dim}
o(M)=-\dim M\qquad \text{and}\qquad  g^{\dim M}(M)=Q_M(1)
\end{equation}
and the Laurent series expansion of $H(M,t)$ around $t=1$ is 
$$
\sum_{n\le  \dim M}\frac{g^n(M)}{(1-t)^n}\,.
$$
In particular, note that 
\begin{equation}
\label{e:dim>}
g^n(M)=0\quad\text{for all}\quad  n>\dim M\,.
\end{equation}

When $R$ is standard graded, meaning that $R$ is generated as a $k$-algebra by $R_1$, we have $Q_M(t)\in \mathbb Z[t,t^{-1}]$, and $Q_M(1)$ is equal to the multiplicity of $M$. The next formula is well-known in this case, and the same proof works when $R$ is not necessarily standard graded, see \cite[Lemma 9]{ABS}:
\begin{equation}
\label{p-decomposition}
g^m(M)=\sum_{\substack{\p\in \Proj{R},\\ \dim R/\p=m}}\lambda_{R_{\p}}(M_\p)g^m(R/\p)\,.
\end{equation}
Here, $m=\dim M$ and $\Proj{R}$ is the set of homogeneous prime ideals of $R$.
\end{bfchunk}

\begin{bfchunk}{New invariants.}
\label{new-inv-notation}
Let $j\geq 0$ be an integer. We introduce notation for the following rational functions: 
\begin{align*}
\phi_R(M,N)(t)&=\frac{H(M, t^{-1})H(N, t)}{H(R, t^{-1})}\,;\\ 
\omega_R^j(M,N)(t)&=\sum_{i=0}^j (-1)^iH(\Ext^i_R(M,N), t)\,;\\
\rho_R^j(M,N)(t)&=\omega^j_R(M,N)(t)-\phi_R(M,N)(t)\,.
\end{align*}
Note that $\phi_R(-,-)$ is biadditive on short exact sequences. 

Let $n$ be an integer. We are interested in the Laurent coefficient 
$$
g^n(\rho^j_R(M,N))=\sum_{i=0}^j (-1)^i g^n(\Ext^i_R(M,N))-g^n(\phi_R(M,N))\,.
$$
Assume $g^n( \Ext^i_R(M,N))=0$ for all $i\gg 0$ and set
\begin{equation}
\label{fn}
\ell_n=\inf\{j\in \mathbb Z\mid  \text{$g^n(\Ext^i_R(M,N))=0$ for all $i>j$}\}.\\
\end{equation}
Further, define 
\begin{equation}
\label{notation-2}
\gamma_R^n(M,N)=g^n(\rho^{\ell_n}_R(M,N))\, ,
\end{equation}
and  note that 
\begin{equation}
\label{notation-1}
\gamma_R^{n}(M,N)=g^n(\rho^{j}_R(M,N))\quad\text{for all $j\ge \ell_n$}\,.
\end{equation}
In particular, if $\ell=\fext{R}(M,N)<\infty$, then  $\dim_R\Ext_R^i(M,N)=0$ for  all $i> \ell$. Thus, if $n>0$, then $g^n(\Ext_R^i(M,N))=0$ for all $i>\ell $, and hence $\ell_n\le \ell$. Hence 
\begin{equation}
\label{gamma-f}
\gamma_R^n(M,N)=g^n(\rho^{\ell}_R(M,N))=g^n(\rho^j_R(M,N))\quad\text{for all $n>0$ and $j\geq \ell$}\,.
\end{equation}

We see from the above that, if $r\ge 0$ is an integer,  $\ell=\fext{R}(M,N)<\infty$ and $j\ge \ell$, then the following are equivalent: 
\begin{numlist}
\item $\gamma^n_R(M,N)=0$ for all $n> r$; 
\item $o(\rho_R^\ell(M,N))\geq -r$;
\item $o(\rho_R^{j}(M,N))\geq -r$. 
\end{numlist}

\end{bfchunk}

\begin{bfchunk}{Shifts.}
\label{phi-shift} For all integers $a,b$ and all $j\ge 0$ we have: 
\begin{align*}
\phi_R\left(M(a), N(b)\right)(t)&=t^{a-b}\phi_R(M,N)(t)\,;\\
\omega^j_R(M(a), N(b))(t)&=t^{a-b}\omega^j_R(M,N)(t)\,;\\
\rho^j_R(M(a),N(b))(t)&=t^{a-b}\rho^j_R(M,N)(t)\,.
\end{align*}
In view of \eqref{order-sum}, the rational functions $\rho^j_R(M(a),N(b))(t)$ and $\rho^j_R(M,N)(t)$ have the same order for any integers $a$ and $b$, since $o(t^{a-b})=0$ per \ref{varphi-general}. 
\end{bfchunk}

\begin{bfchunk}{Vanishing of Ext.}
\label{zero-Exts}
If $\Ext^i_R(M,N)=0$ for all $i>j$, then $\phi_R(M,N)=\omega^j_R(M,N)$ by \cite[Theorem 1]{ABS}, and hence $\rho_R^j(M,N)=0$. 
In particular, if $M$ is free then $\rho_R^j(M,N)=0$ for all $j\ge 1$ and $H(\Hom_R(M,N),t)=\phi_R(M,N)(t)$. 
Also, if $R$ is Gorenstein and $N$ is free, then $\rho_R^j(M,N)=0$ for all $j\ge \dim R$.
\end{bfchunk}

\begin{bfchunk}{Bounds on orders.}
\label{bounds}
We have $\dim_R\Ext^i_R(M,N)\le \dim R/(\ann M+\ann N)$ for all $i$, since $\ann N$ and $\ann M$ both annihilate $\Ext^i_R(M,N)$. Thus by \eqref{add-orders} and \eqref{dim}, it follows that
\begin{equation}
\label{ann-cond}
o(\omega^j_R(M,N))\ge -\dim R/(\ann M+\ann N)\quad\text{for all $j\ge 0$}\,.
\end{equation}
Using \eqref{order-sum} and \eqref{dim}, we see 
\begin{align}
\label{dim-cond}
o(\phi_R(M,N))&=-(\dim M+\dim N-\dim R)\, .
\end{align}
Further, for $u=\dim M+\dim N -\dim R$, one has
\begin{align}
\label{dim-formula}
g^{u}(\phi_R(M,N))&=(-1)^{\dim M-\dim R}\cdot \frac{g^{\dim M}(M)g^{\dim N}(N)}{g^{\dim R}(R)}\, .
\end{align}
The sign in \eqref{dim-formula} is a consequence of the fact that, if we write $\displaystyle H(M,t)=\frac{Q_M(t)}{(1-t)^m}$ for $m=\dim M$ as in \ref{Hilbert series}, then
$$H(M,t^{-1})=\frac{Q_M(t^{-1})}{(1-t^{-1})^m}=\frac{(-1)^mt^mQ_M(t^{-1})}{(1-t)^m}.$$
It follows from \eqref{add-orders}, \eqref{ann-cond}, and \eqref{dim-cond} that
\begin{equation}
\label{dimension-cond}
o(\rho_R^{j}(M,N))\ge -\max\{\dim R/(\ann M+\ann N), \dim M+\dim N-\dim R\}
\end{equation}
for all $j$. In particular, as $\dim R/(\ann M+\ann N)$ and $\dim M + \dim N - \dim R$ are both at most $\min\{\dim M, \dim N\}$, we obtain 
\begin{equation}
\label{min}
o(\rho_R^{j}(M,N))\ge -\min\{\dim M, \dim N\}\ge -\dim R\quad\text{ for all $j$}.
\end{equation}
Since we understand when equality holds in \eqref{add-orders}, we also obtain  
\begin{equation}
\label{equal}
o(\rho_R^{j}(M,N))=-(\dim M+\dim N -\dim R)
\end{equation}
whenever $\dim M+\dim N-\dim R>\dim R/(\ann M+\ann N)$.
\end{bfchunk}

\section{Consequences of regularity conditions}\label{regularity_conditions}
\noindent
We keep the notation of Section \ref{Laurent-coefficients}. For a prime ideal $\p$ of $R$, the \emph{codimension} of $\p$ is $\codim \p=\dim R- \dim R/\p$. If $u$ is an integer, saying that $R$ is \emph{regular in codimension $u$} means that $R_{\p}$ is regular for all $\p\in \Proj{R}$ with $\codim \p \le u$. In this section we continue exploring the Laurent series introduced earlier, with our primary focus on bounds for the order of the series $\rho_R^j(M,N)$ for certain values of $j$, under additional regularity hypotheses.  Propositions \ref{interpret-2} and \ref{uv} are the main results in this section; they will be used to prove parts of Theorems \ref{codimension} and \ref{dimensions}. 

\begin{bfchunk}{A formula for  $\gamma^n_R(M,N)$.}
\label{regular-codim-s-proof}
Let $u$ be an integer with $u<\dim R$ and assume $R$ is regular in codimension $u$. Recall that the notation $\gamma^n_R(M,N)$ was introduced in \eqref{notation-2}. We show here a formula for computing this invariant using the regularity condition. 
We first claim: 
\begin{equation}
\label{Lucho-translate}
g^n(\Ext^i_R(M,N))=0 \text{ for all $n,i$ with $n\ge \dim R-u$ and $i>\dim R-n$.}
\end{equation}
The proof of this statement can be read off  the statement of  \cite[Lemma 10]{ABS}.  Since this proof is rather short, we give it here for the convenience of the reader: 
We set $L=\Ext^i_R(M,N)$ and let $n$, $i$ be as in \eqref{Lucho-translate}. To show $g^n(L)=0$, we need to show $\dim L< n$. Since $n\ge \dim R-u$, it suffices to assume  $\dim L\ge \dim R-u$. We have $\dim L=\dim R/\p$ for some homogeneous prime ideal $\p\in \Supp L$ and $\codim \p=\dim R-\dim L\le u$, so we know that $R_{\p}$ is regular. Since $0\ne L_{\p}=\Ext^i_{R_\p}(M_\p,N_\p)$, we see that $i\le \depth R_\p\le \codim \p=\dim R-\dim L$. Since $i>\dim R-n$, we conclude $\dim L<n$. This finishes the proof of \eqref{Lucho-translate}. 
 
In view of \eqref{Lucho-translate}, if $n\geq \dim R-u$, then the integer $\ell_n$ defined by \eqref{fn} satisfies $\ell_n\le \dim R-n$.  Using  \eqref{notation-2} and \eqref{notation-1} we obtain: 
\begin{equation}
 \label{use-lemma}
\gamma_R^{n}(M,N)=g^n(\rho_R^u(M,N))=g^n(\rho_R^j(M,N))
\end{equation}
for all $n,j$ with $n\ge \dim R-u$ and $j\ge \dim R-n$ and in particular for all $j\ge u$.
\end{bfchunk}

\begin{bfchunk}{Terminology interpretation.}
\label{interpret}
In \cite{ABS}, an invariant $\varepsilon^j_R(M,N)$ was defined for every $j$. Using our notation, this invariant can be defined by the formula
$$
\varepsilon^j_R(M,N)=g^{\dim R-j}(\omega^j_R(M,N))=\sum_{i=0}^j(-1)^ig^{\dim R-j}(\Ext_R^i(M,N))\,.
$$

If $R$ is regular in codimension $u<\dim R$, then, in view of \eqref{use-lemma} we have 
\begin{equation}
\label{interpret-1}
\begin{aligned}
\gamma_R^{\dim R-u}(M,N)&=g^{\dim R-u}(\rho_R^{u}(M,N))\\
&=\varepsilon^u_R(M,N)-g^{\dim R-u}(\phi_R(M,N))\,.
\end{aligned}
\end{equation}
We now interpret the terminology introduced after \cite[Proposition 6]{ABS} using our notation. Set $d=\dim R$ and let $u$ be an integer. In {\it loc.\! cit.\!}  the sequence of coefficients of the Laurent series expansion of $\phi_R(M,N)$  around $t=1$ and the sequence $\{\varepsilon_R^j(M,N)\}_j$ are said to {\it agree up to level} $u$ if 
\begin{equation}
\label{agree}
\phi_R(M,N)(t)=\sum_{j=0}^u\frac{{\varepsilon}^j_R(M,N)}{(1-t)^{d-j}}+O\left (\frac{1}{(1-t)^{d-u-1}}\right)\,.
\end{equation}
Set 
\begin{equation}
\label{ugly}
\varphi(t)=\phi_R(M,N)-\sum_{j=0}^u\frac{{\varepsilon}^j_R(M,N)}{(1-t)^{d-j}}=\phi_R(M,N)-\sum_{n=d-u}^d \frac{\varepsilon_R^{d-n}(M,N)}{(1-t)^n}\,.
\end{equation}
The equation \eqref{agree} holds if and only if $o(\varphi)>-(d-u)$. Since $o(\phi_R(M,N))\ge -d$, this is equivalent to $g^n(\varphi)=0$ for all $n$ with $d-u\le n\le d$. Note that 
\begin{align*}
g^n(\varphi)&=g^n(\phi_R(M,N))-\varepsilon_R^{d-n}(M,N)\\
&=g^n(\phi_R(M,N))-g^n(\omega_R^{d-n}(M,N))\\
&=g^n(\rho_R^{d-n}(M,N))\,.
\end{align*}
 Consequently, the sequence of coefficients of the Laurent expansion of $\phi_R(M,N)$  around $t=1$ and the sequence $\{\varepsilon_R^j(M,N)\}_j$ agree up to level $u$ if and only if 
 $$
 g^n(\rho_R^{d-n}(M,N))=0 \quad\text{for all $n$ with $d-u\le n\le d$}\,.
 $$
\end{bfchunk}

\begin{bfchunk}{Equivalent conditions for bounds on orders.}
\label{new-interpret}
Assume now that $R$ is regular in codimension $u<\dim R$. In view of \eqref{use-lemma}, the following statements are equivalent:  
\begin{numlist}
\item The sequence of coefficients of the Laurent series expansion of $\phi_R(M,N)$ around $t=1$ and the sequence $\{\varepsilon_R^j(M,N)\}_j$ agree up to level $u$; 
\item  $\gamma^{n}_R(M,N)=0$ for all $n\ge \dim R-u$; 
\item $o(\rho_R^u(M,N))>-(\dim R-u)$; 
\item $o(\rho_R^j(M,N))>-(\dim R-u)$ for all $j\ge u$. 
\end{numlist}
If $\ell=\fext{R}(M,N)<\infty$, use \eqref{gamma-f} to see that these conditions are also equivalent to 
\begin{numlist}
\item[(5)] $o(\rho_R^\ell(M,N))>-(\dim R-u)$. 
\end{numlist}
\end{bfchunk}

\begin{bfchunk}{Biadditivity.}
\label{biadditive}
If $R$ is regular in codimension $u<\dim R$, then $\varepsilon^u_R(-,-)$ is biadditive on short exact sequences, see \cite[Lemma 11]{ABS}. Since  $g^{\dim R-u}(\phi_R(-,-))$ is also biadditive, we conclude from \eqref{interpret-1} that $\gamma_R^{\dim R-u}(-,-)$ is biadditive on short exact sequences. 
\end{bfchunk}

In view of the interpretation in \ref{new-interpret}, further translation of \cite[Theorem 7]{ABS} yields:
\begin{proposition}
\label{interpret-2}
 Let $R$ be a graded ring  as in \ref{setting} and let $M$ and $N$ be finitely generated graded $R$-modules. The following hold: 
\begin{numlist}
\item If $R$ has a unique prime $\p$ of codimension $0$ and $R_{\p}$ is a field, then 
$$o(\rho^{0}_R(M,N))>-\dim R.$$
\item If $R$ is an integral domain that is regular in codimension $1$, then 
$$o(\rho^{1}_R(M,N))>-(\dim R-1).$$
\item If $R$ is a unique factorization domain that is regular in codimension $2$, then 
$$o(\rho^{2}_R(M,N))>-(\dim R-2).$$ \qed
\end{numlist}
\end{proposition}

The next  result is inspired by results of  \cite{ABS}  recorded above, and improves on the first inequality in \eqref{min}, under additional hypotheses. 
Let $K$ be the total ring of fractions of $R$. We say an $R$-module $M$ \emph{has rank $r$} if $M\otimes_RK$ is a free $K$-module of rank $r$.  We say $M$ \emph{has a rank} if there is an integer $r$ such that $M$ has rank $r$.

\begin{proposition}
\label{uv}
 Let $R$ be a graded ring as  in \ref{setting}, and assume it is Gorenstein. Let $M$ and $N$ be finitely generated graded $R$-modules.  Set $v=\min\{\dim M, \dim N\}$ and assume $v>0$, $R$  is regular in codimension $\dim R-v$ and, if $v=\dim R$, then $M$ or $N$ has a rank. 

Then $o(\rho_R^{\dim R-v}(M,N))>-v$. 
\end{proposition}

\begin{proof}
Set $d=\dim R$ and $u=d-v$. Since $v>0$, we have $u<d$. 

By \eqref{min}, we know $o(\rho_R^u(M,N))\ge -v$, and hence $g^n(\rho^u_R(M,N))=0$ for all $n>v$. The inequality $o(\rho_R^u(M,N))>-v$ is thus equivalent to $g^v(\rho^u_R(M,N))=0$.  Recall that $\gamma^v_R(M,N)=g^v(\rho^u_R(M,N))$ by \eqref{use-lemma}. Since $R$ is regular in codimension $u$, we know  that $\gamma^v_R(-,-)$ is biadditive on short exact sequences by \ref{biadditive}. 
\begin{Case1}
Assume $\dim N=v$ and  if $v=d$ (that is, $u=0$) then $M$ has a rank. 

If $u>0$, then, since  $R$ is regular in codimension $u$, it is  regular in codimension $1$, hence it is a (normal) domain, cf.~\cite[Theorem 2.2.22]{BH98}, and thus $M$ has a rank in this case as well.

If $M$ is free, then $\rho^u_R(M,N)=0$ by \ref{zero-Exts}, hence $\gamma^v_R(M,N)=0$. In general, we  observe that it suffices to assume that $M$ is maximal Cohen-Macaulay, by replacing $M$ with a high enough syzygy. Indeed, this follows from the fact that the biadditivity of $\gamma_R^v(-,-)$  on short exact sequences gives that  if $M'$ is a syzygy of $M$, then $\gamma^v_R(M,N)=0$ if and only if $\gamma^v_R(M',N)=0$.  The hypothesis that $M$ has a rank is preserved when replacing $M$ with a syzygy. 

Thus we may assume that $M$ is maximal Cohen-Macaulay and has rank $a$. Since $R$ is regular in codimension $u$, we see that $M_{\p}$ is free for all homogeneous prime ideals with $\codim \p\le u$  and hence $(\Ext^i_R(M,N))_{\p}=0$ for all $i>0$ for all such $\p$. This implies $\dim _R\Ext^i_R(M,N)<v$ for all $i>0$, and hence 
$$g^j_R(\Ext^i_R(M,N))=0 \qquad\text{for all $j\ge v$ and all $i>0$.}
$$
by \eqref{e:dim>}. 
 We conclude  $\gamma^v_R(M,N)=g^v(\rho_R^0(M,N))$, see \eqref{notation-2} and \eqref{notation-1}. It suffices thus to show that $g^v(\Hom_R(M,N))=g^v(\phi_R(M,N))$. 

Since $M$ has rank $a$ and $M_\p$ is free for all $\p$ with $\codim \p\le u$, we conclude  $M_{\p}\cong (R_{\p})^a$ for all $\p$ with $\codim \p=0$. Further, since $\dim M=d$, we use  formula \eqref{p-decomposition} to obtain: 
\begin{equation}
\label{M}
\begin{aligned}
g^d(M)&=\sum_{\substack{\p\in \Proj{R},\\ \codim \p=0}}\lambda_{R_\p}(M_\p)g^d(R/\p)\\
&=\sum_{\substack{\p\in \Proj{R},\\ \codim \p=0}}a\lambda_{R_\p}(R_\p)g^d(R/\p) \\
&=ag^d(R)\,.
\end{aligned}
\end{equation}
Using then \eqref{dim-formula}, we  have 
$$
g^v(\phi_R(M,N))=\frac{g^d(M)g^v(N)}{g^d(R)}=ag^v(N)\,.
$$
As $\dim N=v$, we can choose $\p\in \Supp_R N$ with $\codim\p=u$. We have $M_\p\cong (R_\p)^a$, hence $\Hom_{R_\p}(M_\p, N_\p)\cong (N_\p)^a\ne 0$, so  $\p$ is in $\Supp_R(\Hom_R(M,N))$. It follows that $\dim_R(\Hom_R(M,N))\ge v$. We conclude  $\dim_R(\Hom_R(M,N))= v$, since the reverse inequality also holds. We use then \eqref{p-decomposition} again as follows: 
\begin{align*}
g^v(\Hom_R(M,N))&=\sum_{\substack{\p\in \Proj{R},\\ \codim \p=u}}\lambda_{R_\p}(\Hom_{R_\p}(M_\p, N_\p))g^v(R/\p)\\
&=\sum_{\substack{\p\in \Proj{R}, \\ \codim \p=u}}a\lambda_{R_\p}(N_\p))g^v(R/\p)\\
&=ag^v(N)\,.
\end{align*}
We conclude $g^v(\Hom_R(M,N))=g^v(\phi_R(M,N))$, and hence $\gamma^v_R(M,N)=0$. 
\end{Case1}

\begin{Case2} Assume $\dim M=v$ and if $v=d$ (so $u=0$) then $N$ has a rank. 

As in Case 1, we see that $N$ has a rank when $v<d$ as well, since $R$ is a domain in this case. 

If $N$ is free, then we have $\rho_R^j(M,N)=0$ for all $j>\dim R$ by \ref{zero-Exts}, since $R$ is Gorenstein.  Since $\gamma_R^v(M,N)=g^v(\rho_R^j(M,N))$ for all $j\ge u$ by \eqref{use-lemma}, we see that $\gamma_R^v(M,N)=0$ in this case. In  general, we  observe that it suffices to assume that $N$ is maximal Cohen-Macaulay, by replacing $N$ with a high enough syzygy. Indeed, this follows from the fact that the biadditivity of $\gamma_R^v(-,-)$  on short exact sequences gives that  if $N'$ is a syzygy of $N$, then $\gamma^v_R(M,N)=0$ if and only if $\gamma^v_R(M,N')=0$.  The hypothesis that $N$ has a rank is preserved when replacing $N$ with a syzygy. 

Thus we may assume that $N$ is maximal Cohen-Macaulay and has rank $b$. Since $\dim M=v$, we can find a sequence $\underline x$ of length $u$ in $\ann M$ that is regular on $N$. In particular, it follows that $\Ext^i_R(M,N)=0$ for all $i<u$. To show $\gamma^{v}_R(M,N)=0$, it suffices thus to show that $(-1)^ug^v(\Ext^{u}_R(M,N))=g^v(\phi_R(M,N))$. 

The module $M$ admits a finite filtration with subquotients of the form $R/\fq(i)$ with $\fq\in \Proj{R}$ and $\codim  \fq\ge u$. Since $\gamma_R^v(-,-)$ is biadditive on short exact sequences, it suffices to prove $\gamma_R^v(R/\fq(i),N)=0$ for all $i$ and all $\fq\in \Proj{R}$ with $\codim \fq\ge  u$. If $\codim \fq>u$, then $\dim R/\fq(i)<v$, hence \eqref{min}  gives:
$$o(\rho^{j}(R/\fq(i),N))\ge -\min\{\dim R/\q(i), \dim_R N\}=-\dim R/\q(i)\ge -(v-1)
$$ 
for all $j$, and  hence  $\gamma^{v}_R(R/\fq(i) ,N)=0$, cf.~the equivalence (2)$\Leftrightarrow$(3) in \ref{new-interpret}. We may assume thus $\codim \fq=u$, so that $\dim R/\fq(i)=v$, and  $M=R/\fq(i)$.

By \ref{phi-shift}, we have  $o(\rho_R^u(R/\fq(i),N))=o(\rho_R^u(R/\fq,N))$.  It suffices thus to assume $M=R/\fq$ with $\codim \fq=u$.

Since $N$ is maximal Cohen-Macaulay and $R$ is regular in codimension $u$, we know that $N_\p$ is free for all $\p$ with $\codim \p\le u$. Since $N$ has rank $b$, we have then $N_{\p}=(R_{\p})^b$ for all $\p$ with $\codim\p\le u$. 
 Note that $(\Ext^{u}_R(R/\fq,N))_{\fq}=\Ext^{u}_{R_\fq}(k(\fq), (R_{\fq})^b)\cong k(\fq)^b\ne 0$, hence $\fq$ belongs to $\Supp(\Ext^{u}_R(R/\fq,N))$. Since we have $\dim_R(\Ext^{u}_R(R/\fq,N))\le v$, we conclude $\dim_R(\Ext^{u}_R(R/\fq,N))=v$. 

We have: 
\begin{align*}
g^v(\Ext^{u}_R(R/\fq,N))&=\sum_{\substack{\p\in \Proj{R},\\ \codim \p=u}}\lambda_{R_\p}(\Ext^u_{R_\p}((R/\fq)_\p, N_\p))g^v(R/\p)\\
&=b\lambda_{R_\fq}(\Ext^{u}_{R_\fq}(k(\fq), R_{\fq}))g^v(R/\fq)\\
&=bg^v(R/\fq)\,.
\end{align*}
On the other hand, $g^d(N)=bg^d(R)$ by an argument similar  to that in \eqref{M}. Using further \eqref{dim-formula}, we have 
\begin{align*}
g^v(\phi_R(R/\fq,N))&=(-1)^u\frac{g^v(R/\fq)g^d(N)}{g^d(R)}=(-1)^u bg^v_R(R/\fq)\\&=(-1)^ug^v(\Ext^{u}_R(R/\fq,N))
\end{align*}
and this implies $\gamma^v_R(R/\fq,N)=0$, completing the proof of Case 2.
\end{Case2}
Finally, note there are no remaining cases: if $v=d$ then $\dim M=\dim N = d$, and so if either $M$ or $N$ has a rank, then one of the above cases apply.
\end{proof}

\section{Reducing complexity}
\label{reducing-complexity}
\noindent
As indicated by the title, the purpose of this section is to establish results which will later allow us to craft induction arguments by reducing the complexity of a pair of modules. 

We maintain the setting in  Section \ref{Laurent-coefficients}, and assume that $R$ is a graded complete intersection and that $M$ and $N$ are finitely generated graded $R$-modules. In particular, $R=Q/(f_1,...,f_c)$ where $Q=k[x_1,\dots, x_\nu]$ is a polynomial ring over a field $k$ with variables $x_i$ of positive degree and $f_1,\dots,f_c$ is a homogeneous $Q$-regular sequence. Let $e_i=\deg(f_i)$ for each $i$.

\begin{bfchunk}{Graded Eisenbud operators.}
\label{notation} 
Let $(F, \partial)$ be a minimal graded free resolution of $M$ over the ring $R$:
$$
(F,\partial)= \quad \dots \to  F_i\xrightarrow{\partial_i} F_{i-1}\to \dots \to F_1\xrightarrow{\partial_1} F_0
$$
and for each $i\ge 0$ set $\syz{i}{M}=\Coker(\partial_{i+1})$, the $i$th syzygy of $M$. There are exact sequences
\begin{equation}
\label{syz-ses}
0\to \syz{i}{M}\to F_{i-1}\to \syz{i-1}{M}\to 0\,.
\end{equation}

We consider the Eisenbud operators $\tau^1, \dots, \tau^c$, as constructed in Eisenbud's \cite{Eis80}, see also \cite[Construction 9.1.5]{Av}. In our graded setting, each operator $\tau^n$ is a chain map on  $F$ of bidegree $(-2,-e_n)$:
$$
\tau^n_j\col F_{j+2}\to F_{j}(-e_n)\,.
$$
These are constructed as follows: Let $(\widetilde F, \widetilde \partial)$ be a lifting of $F$ to $Q$, so that $(F,\partial)=(\widetilde F\otimes_QR, \widetilde\partial\otimes_QR)$. The relation $\partial^2=0$ yields $\widetilde\partial ^2(\widetilde F)\subseteq (f_1, \dots, f_c)\widetilde F$, hence we can choose homological degree $-2$ graded endomorphisms $\widetilde \tau^n\colon \widetilde F\to \widetilde F(-e_n)$ such that $\wt \partial^2=\sum_{n=1}^{c} f_n\wt \tau^n$. Set $\tau^n=\wt\tau^n\otimes_QR$. 

The Eisenbud operators determine maps of graded $R$-modules 
$$
\chi_n^j\colon \Ext_R^{j}(M,N)(e_n)\to  \Ext^{j+2}_R(M,N)\,.
$$
These maps turn $\bigoplus_{j,i}\Ext^j_R(M, N)_i$ into a bigraded module over a  bigraded polynomial  ring $R[\chi_1, \dots, \chi_{c}]$ with variables $\chi_n$ of bidegree $(2,-e_n)$. Furthermore, $\bigoplus_{j,i}\Ext^j_R(M, N)_i$ is a finitely generated $R[\chi_1,\dots,\chi_{c}]$-module, as proved by Gulliksen \cite{Gu}, see also \cite[Theorem 9.1.4]{Av}. 
\end{bfchunk}

The following lemma is a standard argument on the existence of superficial elements. We spell out a proof in order to properly address the fact that, in our context, such elements can be chosen to be bihomogeneous.  
\begin{lemma}
\label{superficial-element}
Assume that $k$ is infinite. There exist then an integer $e\ge 0$,  a bihomogeneous element $\chi\in R[\chi_1, \dots, \chi_{c}]$ of bidegree $(2,-e)$ and an integer $s\ge 0$ such that $\chi$ is a non-zerodivisor on $\Ext^{\geqslant s}_R(M,N)$. 
\end{lemma}

\begin{proof}
Set $X= \bigoplus_{j,i}\Ext^j_R(M, N)_i$ and $\mathcal R=R[\chi_1, \dots, \chi_{c}]$. By the previous comments, $X$ is a bigraded finitely generated $\mathcal R$-module. We use here a primary decomposition argument. Since $\mathcal R$ is bigraded by the monoid $\mathbb{N}\times -\mathbb{N}$, all ideals in this proof are bihomogeneous. We refer to Northcott \cite[Section 2.13]{Nor} for a treatment of primary decomposition in a bigraded setting. 
Choose a primary decomposition
$$
0=\q_1'\cap\dots  \cap \q'_m\cap \q_1\cap \dots \cap \q_n
$$
of the zero submodule of $X$, where the associated primes $\p_i'$ of $\q_i'$ do not contain $\mathcal R_{2,*}$ and the associated primes $\p_i$ of $\q_i$ contain $\mathcal R_{2,*}$. Set $S_i=\p_i'\cap \mathcal R_{2,*}$. Since each $S_i$ is properly contained in $\mathcal R_{2,*}$, Nakayama's Lemma shows that $(S_i+\fm \mathcal R_{2,*})/\fm \mathcal R_{2,*}$ is a proper subspace of $\mathcal R_{2,*}/\fm \mathcal R_{2,*}$. Since $k$ is infinite, we have $\cup_{i=1}^m (S_i+\fm \mathcal R_{2,*})/\fm \mathcal R_{2,*}\ne \mathcal R_{2,*}/\fm \mathcal R_{2,*}$.  There exists thus an integer $e\ge 0$ and an element $\chi$ of bidegree $(2,-e)$  such that $\chi\notin S_i$ for all $i$, hence $\chi\notin \p'_i$ for all $i$. Since $\q_i'$ is primary to $\p_i'$ for each $i$, it follows that $(0:_X\chi)\subseteq \q_i'$ for each $i$.  

Since each $\p_i$ is finitely generated and $\p_i=\sqrt{\ann(X/\q_i)}$, there exists $a_i$ such that $\p_i^{a_i}\subseteq \ann(X/\q_i)$. As $\mathcal R_{2,*}\subseteq \p_i$, we have $(\mathcal R_{2,*})^{a_i}\subseteq \ann(X/\q_i)$ for each $i$. Choose $a$ large enough such that 
$$(\mathcal R_{2,*})^{a}\subseteq \ann(X/\q_1)\cap \dots \cap \ann(X/\q_n)$$
and hence 
$$
(\mathcal R_{2,*})^{a}X\subseteq \q_1\cap \dots\cap \q_n\,.
$$
Since $X$ is finitely generated and the generators  $\chi_i$ of $\mathcal R_+$ have bidegree $(2,-e_i)$, for $j\gg 0$ we have $X^{j,*}\subseteq (\mathcal R_{2,*})^aX$ for $j$ sufficiently large. Thus 
$$(0:_X\chi)\cap X^{j,*}\subseteq (0:_X\chi)\cap (\mathcal R_{2,*})^{a}X\subseteq \q_1'\cap\dots  \cap \q'_m\cap \q_1\cap \dots \cap\q_n=0
$$
for $j\gg0$.
\end{proof}

\begin{bfchunk}{Convention for the case when $\cx_R(M,N)=1$.}
\label{complexity-1}
Assume $\fext{R}(M,N)<\infty$. If $\cx_R(M, N)=1$, then there exists an integer $t$ such that  $\lambda(\Ext_R^i(M,N))=\lambda(\Ext_R^{i+2}(M,N))$ for all $i\ge t$; see \cite[Theorem 5.6(4)]{AB}.  Assume also that $k$ is infinite. In this case, we make the convention that the integer $s$ in Lemma \ref{superficial-element} is chosen so that $s\ge t$. This ensures that the element $\chi$ coming from the lemma induces an isomorphism  $\Ext_R^i(M,N)(e)\to \Ext_R^{i+2}(M,N)$ for all $i\ge s$. 
\end{bfchunk}

\begin{bfchunk}{Construction.}
\label{construction}
Assume $\cx_R(M,N)>0$ and $k$ is infinite. Let $\chi$ and $s$ be as in Lemma \ref{superficial-element}.   Let $n>s$. Since $\chi$ has cohomological degree $2$, it is a linear combination of the elements $\chi_i$, and it comes thus from a chain map $\tau$ on $F$ with 
$$
\tau_j\col F_{j+2}\to F_{j}(-e)
$$
for all $j\ge 0$. As in the proof of \cite[Proposition 2.2(i)]{Bergh}, with the added observation that all maps involved are maps of graded $R$-modules, this chain map induces  a map $\syz{n+2}{M}\to \syz{n}{M}(-e)$ and a commutative pushout diagram with exact rows
\begin{equation}
\label{pushout}
\xymatrix{
0\ar@{->}[r]&\syz{n+2}{M}\ar@{->}[r]\ar@{->}[d]&F_{n+1}\ar@{->}[r]\ar@{->}[d]&\syz{n+1}{M}\ar@{->}[r]\ar@{=}[d]&0\\
0\ar@{->}[r]&\syz{n}{M}(-e)\ar@{->}[r]&K\ar@{->}[r]&\syz{n+1}{M}\ar@{->}[r]&0}
\end{equation}
such that the connecting homomorphisms in the long exact sequence in homology induced by the bottom row  are given by multiplication by $\chi$: 
\[\xymatrix{
\Ext^{i-1}_R(\syz{n}{M}(-e),N)\ar[r]^-{\chi} &\Ext^i_R(\syz{n+1}{M},N)\\
\Ext^{i+n-1}_R(M,N)(e)\ar@{=}[u] & \Ext^{i+n+1}_R(M,N)\ar@{=}[u]
}\]
for all $i>1$. Since $n\ge s+1$, these maps are injective for all $i>1$. 

Therefore, for each $n>s$ there exists a graded $R$-module $K$ (depending on $n$)  such that we have exact sequences
\begin{equation}
\label{Hom-ses}
\xymatrix@C=1em{
0\ar[r] & \Hom_R(\syz{n+1}{M}, N)\ar[r] & \Hom_R(K,N)\ar[r] &  \Hom_R(\syz{n}{M}, N)(e)
\ar@{->} `r[d] `[l] `[dlll] `[l] [dll] \\ 
& \Ext^{n+2}_R(M,N)\ar[r] &  \Ext^1_R(K,N)\ar[r] & 0
}
\end{equation}
and, for all $i>1$,
\begin{equation}
\label{Ext-ses}
0\to \Ext_R^{i+n-1}(M,N)(e)\to \Ext^{i+n+1}_R(M,N)\to \Ext^i_R(K,N)\to 0\,.
\end{equation}
When $\fext{R}(M,N)<\infty$ and $\cx_R(M,N)=1$, with the convention in \ref{complexity-1} in place, we get 
\begin{equation}
\label{exts}
\Ext_R^i(K,N)=0\quad \text{and} \quad \Ext^{i+n-1}_R(M,N)\cong \Ext_R^{i+n+1}(M,N)
\end{equation}
for all $i>1$.
\end{bfchunk}

\begin{bfchunk}{Reducing complexity.} 
\label{reduce}
Assume $\ell=\fext{R}(M,N)<\infty$ and $k$ is infinite. Choose $s$ as in Lemma \ref{superficial-element} and also \ref{complexity-1}, and let $n>\max\{s,\ell\}$. Let $K$ be the $R$-module (depending on $n$) constructed in \ref{construction}.  Additionally, assume $\cx_R(M,N)>0$. 

A length count in the exact sequence \eqref{Ext-ses} gives $\fext{R}(K,N)\le 1$ and 
\begin{equation}
\label{beta}
\beta_i^R(K,N)=\beta_{i+n+1}^R(M,N)-\beta_{i+n-1}^R(M,N)
\end{equation}
for all $i>1$, yielding equalities
\begin{align}
\label{cx-reduction}
\cx_R(K,N)&=\cx_R(M,N)-1\\
\label{h-reduction}
h_R(K,N)&= 2(-1)^{n+1}(\cx_R(M,N)-1)h_R(M,N) \,.
\end{align}
 We note that \eqref{cx-reduction} follows directly from the definition of complexity. We give below an explanation for \eqref{h-reduction}.  Indeed,  set $r=\cx_R(M,N)$. As described in Section \ref{Laurent-coefficients}, there exists an integer $i_0$ such that 
\begin{align*}
\beta_{2i}^R(M,N)&=a_{r-1}(2i)^{r-1}+\dots +a_1(2i)+a_0 \\
\beta_{2i+1}^R(M,N)&=b_{r-1}(2i+1)^{r-1}+\dots +b_1(2i+1)+b_0
\end{align*}
for all $i\ge i_0$, and $h_R(M,N)=a_{r-1}-b_{r-1}$.  If $n$ is even, then using \eqref{beta}, we have for all $j>i_0$: 
\begin{align*}
\beta_{2j}^R(K,N)&=b_{r-1}\left((2j+n+1)^{r-1}-(2j+n-1)^{r-1}\right)+\cdots\\
&=2(r-1)b_{r-1}(2j)^{r-2}+\cdots\,,\quad\text{and} \\
\beta_{2j+1}^R(K,N)&=a_{r-1}\left((2j+n+2)^{r-1}-(2j+n)^{r-1}\right)+\cdots\\
&=2(r-1)a_{r-1}(2j+1)^{r-2}+\cdots\,. 
\end{align*}
These equalities show 
$$h_R(K,N)=2(r-1)(b_{r-1}-a_{r-1})=-2(r-1)h_R(M,N)\,.$$
If $n$ is odd, we similarly obtain
$$h_R(K,N)=2(r-1)(a_{r-1}-b_{r-1})=2(r-1)h_R(M,N)\,.$$
\end{bfchunk}

We now set the stage for the proof of Theorem 2 in the introduction.

\begin{proposition}
\label{formula}
Adopt the setting in {\rm \ref{reduce}}. The following holds: 
\begin{align*}
\rho^1_R&(K,N)(t)=\\
&(-1)^{n}(t^{-e}-1)\rho_R^{n}(M,N)(t)+H({\Ext^{n+1}_R(M,N)},t)-H({\Ext^{n+2}_R(M,N)},t)\,.
\end{align*}
\end{proposition}

\begin{proof}
Identifying $\Ext^1_R(\syz{i-1}{M},N)$ with $\Ext^{i}_R(M,N)$, the short exact sequence \eqref{syz-ses} gives an exact sequence 
\[\xymatrix@C=.5em{
0\ar[r] & \Hom_R(\syz{i-1}{M},N)\ar[r] & \Hom_R(F_{i-1},N)\ar[r] & \Hom_R(\syz{i}{M},N) \ar[r] & \Ext^i_R(M,N)\ar[r] &0\,. 
}\]
Recall we have an equality  $H(\Hom_R(F_{i-1}, N),t)=\phi_R(F_{i-1},N)(t)$, see  \ref{zero-Exts}. In view of this equality, a Hilbert series count on this exact sequence gives
\begin{equation}
\label{syz-sum}
\begin{aligned}
H&({\Hom_R(\syz{i}{M},N)}, t)=\\
&-H({\Hom_R(\syz{i-1}{M},N)},t)+H({\Ext^{i}_R(M,N)},t)+\phi_R(F_{i-1},N)(t)\,.
\end{aligned}
\end{equation}
As $\phi_R(-,-)$ is biadditive on short exact sequences, \eqref{syz-ses} gives
\begin{equation}
\label{phi-syz}
{\phi_R(\syz{i}{M},N)}(t)+\phi_R(\syz{i-1}{M},N)(t)=\phi_R(F_{i-1},N)(t)\,.
\end{equation}
Combining  \eqref{phi-syz} and \eqref{syz-sum}, we have: 
\begin{equation}
\label{combine}
\begin{aligned}
H&(\Hom_R(\syz{i}{M},N),t)-{\phi_R(\syz{i}{M},N)}(t)=\\
&-\left(H({\Hom_R(\syz{i-1}{M},N)},t)-\phi_R(\syz{i-1}{M},N)(t)\right)+H({\Ext^{i}_R(M,N)},t)\,.
\end{aligned}
\end{equation}
Applying this repeatedly, starting with $i=n$ and ending with $i=1$, we obtain: 
\begin{equation}
\label{rewrite-phi}
H(\Hom_R(\syz{n}{M},N),t)-\phi_R(\syz{n}{M},N)(t)=(-1)^{n}\rho_R^{n}(M,N)(t)\,.
\end{equation}
The exact sequence \eqref{Hom-ses} gives
\begin{equation}
\label{Hom-sum}
\begin{aligned}
H&(\Hom_R(K,N),t)-H({\Ext^1_R(K,N)},t)=\\
&H(\Hom_R(\syz{n+1}{M},N),t)+t^{-e}H({\Hom_R(\syz{n}{M},N)},t)-H({\Ext^{n+2}_R(M,N)},t)\,.
\end{aligned}
\end{equation}
In addition, the bottom short exact sequence in the diagram  \eqref{pushout} gives
 \begin{equation}
 \label{pushout-comput}
 \phi_R(K,N)(t)=\phi_R(\syz{n+1}{M},N)(t)+t^{-e}\phi_R(\syz{n}{M},N)(t)\,.
 \end{equation} 
Combining \eqref{Hom-sum} and \eqref{pushout-comput}, we have 
\begin{align*}
H(\Hom_R(K,N),t)-&H(\Ext^1_R(K,N),t)-\phi_R(K,N)(t)=\\
&H(\Hom_R(\syz{n+1}{M},N),t)-\phi_R(\syz{n+1}{M},N)(t)+\\
&t^{-e}\left(H(\Hom_R(\syz{n}{M},N),t)-\phi_R(\syz{n}{M},N)(t)\right)-\\
&H(\Ext_R^{n+2}(M,N),t)\,.
\end{align*}
The left-hand side of this equation is $\rho^1_R(K,N)$. Using also \eqref{combine}, this yields 
\begin{align*}
\rho^1_R(K,N)&=(t^{-e}-1)\left(H(\Hom_R(\syz{n}{M},N),t)-\phi_R(\syz{n}{M},N)(t)\right)+\\
&\qquad\qquad \qquad\qquad\qquad H({\Ext^{n+1}_R(M,N)},t)-H(\Ext_R^{n+2}(M,N),t)\,.
\end{align*}
Finally, using \eqref{rewrite-phi} to rewrite the right-hand side of this expression produces the sought after equality.
\end{proof}

We end the section by pointing out that the assumption that $k$ is infinite is not essential to the arguments elsewhere in this paper.
\begin{remark}
\label{infinitefield}
If $k$ is not infinite, we may consider a faithfully flat ring extension $R\to R'$, where $R'=k'[x_1,\dots, x_\nu]/(f_1,...,f_c)$ for an infinite field extension $k'$ of $k$. Let $M'=R'\otimes_RM$ and $N'=R'\otimes_R N$. Length is invariant under faithfully flat ring extensions, so that $\fext{R}(M,N)=\fext{R'}(M',N')$ and $\beta_i^R(M,N)=\beta_i^{R'}(M',N')$ for all $i\geq \fext{R}(M,N)$. Thus $\cx_R(M,N)=\cx_{R'}(M',N')$ and $h_{R}(M,N)=h_{R'}(M',N')$. Furthermore, the Hilbert series of a given $R'$-module is the same as its Hilbert series when considered as an $R$-module, so in particular, $\rho_{R'}^j(M',N')=\rho_{R}^j(M,N)$ for any integer $j\geq 0$.
\end{remark}

\section{Vanishing of $h_R(M,N)$}
\label{criterion-vanishing}
\noindent
In this section we prove the theorems announced in the introduction.

\begin{theorem}
\label{criterion}
 Let $R$ be a graded complete intersection and let $M$ and $N$ be finitely generated graded $R$-modules with $\fext{R}(M,N)<\infty$ and $\cx_R(M,N)\ge 1$.  
Set $r=\cx_R(M,N)$ and $\ell=\fext{R}(M,N)$. 

The following hold: 
\begin{numlist}
\item $o(\rho^\ell_R(M,N))\ge -r$, and
\item $h_R(M,N)=0$ if and only if  $o(\rho_R^\ell(M,N))>-r$. 
\end{numlist}
\end{theorem}

\begin{proof}
In view of Remark \ref{infinitefield}, we may assume that the underlying field is infinite, so that the results of Section \ref{reducing-complexity} apply. We choose the integers $n$, $s$ and the $R$-module $K$ as in \ref{reduce}. 
 
The equivalence $(2)\Leftrightarrow(3)$ in \ref{new-inv-notation} shows that the inequality $o(\rho_R^\ell(M,N))\ge -r$ is equivalent to $o(\rho_R^{n}(M,N))\ge -r$ and also, since $r\geq 1$, that the inequality $o(\rho_R^\ell(M,N))>-r$ is equivalent to  $o(\rho_R^{n}(M,N))>-r$.
 
The proof relies on using the formula in Proposition \ref{formula}. Set 
$$
A(t)=H(\Ext^{n+1}_R(M,N),t)-H(\Ext^{n+2}_R(M,N),t)\,.
$$
With this notation, the formula becomes: 
\begin{equation}
\label{formula-ref}
\rho^1_R(K,N)(t)=(-1)^{n}(t^{-e}-1)\rho_R^{n}(M,N)(t)+A(t)\,.
\end{equation}
Note that $o(t^{-e}-1)=1$.  Since $n> \ell$, we have $A(t)\in \mathbb Z[t,t^{-1}]$,  so  $o(A)\ge 0$, and $o(A)\ge 1$ if and only if $A(1)=0$, cf. \ref{varphi-general}. 

We prove both statements by induction on $r$. 

First assume $r=1$. We have then $\rho^1_R(K,N)=0$ by \ref{zero-Exts} and \eqref{exts}, hence formula \eqref{formula-ref} becomes
$$
(-1)^{n}(t^{-e}-1)\rho_R^{n}(M,N)(t)=-A(t)\,.
$$
Since $o(t^{-e}-1)=1$ and $o(A)\ge 0$,  we see from this that $o(\rho^{n}_R(M,N))\ge -1$. This proves the base case of (1). For (2), note that $o(\rho^{n}_R(M,N))\ge 0$ if and only if $o(A)\ge 1$, and that  $o(A)\ge 1$ if and only if $A(1)=0$. Thus (2) holds for $r=1$ if and only if $\beta_{n+1}^R(M,N)=\beta_{n+2}^R(M,N)$. In view of the isomorphisms in \eqref{exts}, this equality is equivalent to $\beta_j^R(M,N)=\beta_{j+1}^R(M,N)$ for all $j\gg0$, or equivalently, $h_R(M,N)=0$. Thus the base case of (2) also holds.

Next assume $r>1$ and that the statements hold for pairs of modules of complexity $r-1$. By \ref{reduce}, we know  $\cx_R(K,N)=r-1$ and ${\ell'}=\fext{R}(K,N)\le 1$. The induction hypothesis gives that $o(\rho_R^{\ell'}(K,N))\ge -r+1$,  and $o(\rho_R^{\ell'}(K,N))>-r+1$ if and only if 
$h_R(K,N)=0$. Using the equivalence $(2)\Leftrightarrow(3)$ in \ref{new-inv-notation} as in the beginning of the proof, we can replace $o(\rho_R^{\ell'}(K,N))$ by $o(\rho_R^1(K,N))$ in these statements. 

Using then 
\eqref{order-sum}, \eqref{add-orders} and \eqref{formula-ref}, and the fact that $o(A)\ge 0$ and $o(t^{-e}-1)=1$, we have: 
\begin{align*}
1+o(\rho_R^n(M,N))&=o\left((-1)^n(t^{-e}-1)\rho_R^n(M,N)\right)=o\left(\rho_R^1(K,N))-A(t)\right)\\
&\ge \min\{o(\rho_R^1(K,N)), A(t)\}\ge -r+1
\end{align*}
and hence $o(\rho_R^n(M,N))\ge  -r$.  Similar arguments show  that $o(\rho_R^1(K,N))>-r+1$ if and only if $o(\rho_R^n(M,N))>-r$. Since $h_R(M,N)=0$ if and only if $h_R(K,N)=0$ by \eqref{h-reduction}, this completes the proof of (2).
\end{proof}

We now prove our main results on the vanishing of $h_R(M,N)$ for pairs of modules with large complexity relative to dimension; the first gives bounds that depend on $\dim R$, while the bounds in Theorem \ref{dimensions} depend on $\dim M$ and $\dim N$  as well. 

\begin{theorem}\label{codimension} Let $R$ be a graded complete intersection and let $M$ and $N$ be finitely generated graded $R$-modules such that $\fext{R}(M,N)<\infty$ and $\cx_R(M,N)\ge 1$. 

The equality $h_R(M,N)=0$ holds under any of the following conditions: 

\begin{numlist}
\item $\cx_R(M,N)>\dim R$.
\item $\cx_R(M,N)= \dim R$ and $R$ has a unique prime ideal $\p$ of codimension $0$ and $R_\p$ is a field.
\item $\cx_R(M,N)= \dim R-1$ and the ring $R$ is regular in codimension $1$.
\item $\cx_R(M,N)=\dim R-2$ and $R$ is a unique factorization domain which is regular in codimension $2$.
\end{numlist}
\end{theorem}

\begin{proof}
Set $r=\cx_R(M,N)$ and $\ell=\fext{R}(M,N)$.  By Theorem \ref{criterion}, it is sufficient to show $o(\rho_R^\ell(M,N))>-r$. 

The inequality $o(\rho_R^{\ell}(M,N))>-r$ follows from \eqref{min} under the hypothesis on complexity in (1). The remaining parts (2), (3), and (4) follow from \cite{ABS}, as interpreted in Proposition \ref{interpret-2} in view of \ref{new-interpret}.
\end{proof}

Finally, a proof of Theorem 1 from the introduction is contained in following:
\begin{theorem} 
\label{dimensions}
Let $R$ be a graded complete intersection and let $M$ and $N$ be finitely generated graded $R$-modules such that $\fext{R}(M,N)<\infty$ and $\cx_R(M,N)\ge 1$. 

The equality $h_R(M,N)=0$ holds under any of the following conditions: 
\begin{numlist}
\item  $\cx_R(M,N)> \max\{\dim R/(\ann M+\ann N), \dim M+\dim N-\dim R\}$.
\item  $\cx_R(M,N)> \min\{\dim M, \dim N\}$.
\item $\cx_R(M,N)=\min\{\dim M, \dim N\}$, the ring $R$ is regular in codimension $\dim R-\cx_R(M,N)$, and, if $\dim R=\cx_R(M,N)$, then $M$ or $N$ has a rank. 
\end{numlist}
\end{theorem}
\begin{proof}
Set $r=\cx_R(M,N)$ and $\ell=\fext{R}(M,N)$.  By Theorem \ref{criterion}, it is sufficient to show $o(\rho_R^\ell(M,N))>-r$. 
Under the hypotheses in (1) and (2), this follows from \eqref{dimension-cond} and \eqref{min}. Under the assumptions in (3), it follows from Proposition \ref{uv} 
that $o(\rho_R^{\dim R-r}(M,N))>-r$, and the conclusion follows from the equivalence (3)$\Leftrightarrow$(5) in \ref{new-interpret}, with $u=\dim R-r$. 
\end{proof}

The following corollary gives a cohomological extension, in the graded setting, of Hochster's result for hypersurfaces mentioned in the introduction. 

\begin{corollary}\label{Hochster_corollary}
 Let $R$ be a graded complete intersection and let $M$ and $N$ be finitely generated graded $R$-modules with $\cx_R(M,N)>\dim R/(\ann M+\ann N)$ and $\fext{R}(M,N)<\infty$. The following are then equivalent: 
 \begin{enumerate}[\quad\rm(1)]
 \item $h_R(M,N)=0$;
 \item $\dim M+\dim N-\dim R<\cx_R(M,N)$. 
 \end{enumerate}
\end{corollary}
\begin{proof}
(2)$\implies$(1) follows directly from Theorem \ref{dimensions} (1). 

Assume now $h_R(M,N)=0$ and $\dim M+\dim N-\dim R\ge \cx_R(M,N)$. The hypothesis implies $\dim M+\dim N-\dim R>\dim R/(\ann M+\ann N)$. In view of \eqref{equal}, we have 
$$
o(\rho_R^{j}(M,N))=-(\dim M+\dim N -\dim R) \quad \text{for all $j\ge 0$.}
$$
and Theorem \ref{criterion} (2) gives 
$$
-(\dim M+\dim N -\dim R)>-\cx_R(M,N)\,
$$
which is a contradiction with our assumption. Hence (2) must hold. 
\end{proof}

\begin{remark}\label{Hochster_remark}
In the case $\cx_R(M,N)=0$, then $h_R(M,N)=0$ automatically holds, and moreover, $\dim M+\dim N-\dim R\le \dim R/(\ann M+\ann N)$ also holds. Indeed, if $\dim M + \dim N- \dim R> \dim R/(\ann M+\ann N)$, then in view of \eqref{equal}, we have $o(\rho_R^{j}(M,N))=-(\dim M+\dim N -\dim R)$ for all $j\ge 0$, contradicting \ref{zero-Exts}. 
\end{remark}

\section{Pairs of modules with arbitrary complexity}\label{arbitrary_complexity}
\noindent
Given a local complete intersection $R$ of codimension $c$ and with residue field $k$, it is known that for any integer $r$ between $1$ and $c$ there exists an $R$-module $M$ of complexity $r$, see Avramov, Gasharov, and Peeva \cite[(5.7)]{AGP}.  Consequently, one has $\cx_R(M,k)=\cx_R(M)=r$. One can in fact construct pairs of modules of any given complexity without having to assume that one of the modules is the residue field. We present below a more general construction that makes this point, and also give an example that shows the additional hypotheses in Theorem \ref{dimensions}(3) are necessary. 

\begin{proposition}
\label{Dave} 
Let $Q$ be a ring, let $f_1,...,f_c$ be a $Q$-regular sequence, and set $R=Q/(f_1,...,f_c)$. 
Assume that one of the following holds: 
\begin{numlist}
\item $Q$ is a regular local ring with maximal ideal $\fn$ and residue field $k$, and $f_1,...,f_c\in \fn^2$, or
\item $Q$ is a polynomial ring over a field $k$ with homogeneous maximal ideal $\fn$, and $f_1,...,f_c$ is a homogeneous sequence in $\fn^2$.
\end{numlist}
Set $R_1=Q/(f_1,\dots,f_i)$ and $R_2=Q/(f_j,\dots,f_c)$ with $1\le i,j\le c$ and consider the syzygy modules $\widetilde{M}=\syz[R_1]{\dim R_1}{k}$ and $\widetilde{N}=\syz[R_2]{\dim R_2}{k}$.  

The $R$-modules $M=\widetilde M/(f_{i+1},\dots,f_c)\widetilde M$ and $N=\widetilde N/(f_1,\dots,f_{j-1})\widetilde N$ satisfy $\fext{R}(M,N)<\infty$. If $j\le i$, then $\cx_R(M,N)=i-j+1$, otherwise $\cx_R(M,N)=0$.
\end{proposition}

\begin{proof} 
Assume first that $(Q,\fn,k)$ is a regular local ring and $f_1,...,f_c\in \fn^2$ is a $Q$-regular sequence. Since $\widetilde M$ is a syzygy of the residue field of $R_1$, it has maximal complexity over $R_1$, namely $i$.  Moreover, as $\widetilde{M}$ is a $\dim R_1$-th syzygy over $R_1$, a standard depth lemma shows that $\widetilde{M}$ is maximal Cohen-Macaulay as an $R_1$-module, and so $f_{i+1},\dots,f_c$ is regular on $\widetilde M$. Similarly, $\widetilde N$ is an $R_2$-module of complexity $c-j+1$ such that $f_1,\dots,f_{j-1}$ is regular on $\widetilde N$.  

Since $f_{i+1},\dots,f_c$ is regular on $\widetilde M$ and annihilates $N$, we have the standard isomorphisms $\Ext_R^{n}(M,N)\cong\Ext^{n}_{R_1}(\widetilde M,N)$ for all ${n}\geq 0$, see for example \cite[Lemma 2, p. 140]{Mat89}. As $\widetilde M$ is a syzygy of the residue field of $R_1$, for any non-maximal prime $\mathfrak p$ of $R_1$ we have that 
$\widetilde M_{\mathfrak p}$ is a free $(R_1)_{\mathfrak p}$-module.  Thus for any non-maximal prime $\mathfrak p$ of $R$ we obtain $\Ext_R^{n}(M,N)_{\mathfrak p}=\Ext^{n}_{(R_1)_{\mathfrak p}}(\widetilde M_{\mathfrak p},N_{\mathfrak p})=0$
for all $n>0$, and so $\fext{R}(M,N)=0$.

The remainder of the proof uses results and notation found in \cite{BJ}. In particular, 
$\V_R(M,N)$ is the support variety of the pair $(M,N)$ over $R$, and $\V_R(M)=\V_R(M,k)$ is that of $M$. Recall from \cite{AB} that the dimension of the support variety $\V_R(M,N)$ is the complexity (as defined in \cite{AB}) of the 
pair $(M,N)$.  Set $I=(f_1,\dots,f_c)$, $J_1=(f_1,\dots,f_i)$, and $J_2=(f_j,\dots,f_c)$, and let
$\varphi_{J_1}:J_1/\fn J_1 \to I/\fn I$ and $\varphi_{J_2}:J_2/\fn J_2 \to I/\fn I$ denote the natural maps.

By assumption we have $\V_{R_1}(\widetilde M)=J_1/\fn J_1$ and
$\V_{R_2}(\widetilde N)=J_2/\fn J_2$.  By \cite[Corollary 5.2]{BJ} we therefore have
$\V_R(M)=\varphi_{J_1}(J_1/\fn J_1)$ and $\V_R(N)=\varphi_{J_2}(J_2/\fn J_2)$.  Putting this together with 
\cite[Theorem 5.6(8)]{AB} we find that
\begin{align*}
\V_R(M,N)&=\V_R(M)\cap \V_R(N) \\
&=\varphi_{J_1}(J_1/\fn J_1)\cap\varphi_{J_2}(J_2/\fn J_2) 
\end{align*}
is a linear subspace of dimension $i-j+1$ if $j\le i$, and zero otherwise. This is the complexity of the pair $(M,N)$ as defined in \cite{AB}.  As these two notions of complexity agree \cite[Proposition 2.2]{CD}, we obtain the desired result in the local case.

Finally, assume instead that $Q$ is a polynomial ring over a field $k$, with homogeneous maximal ideal $\fn$, and that $f_1,\dots,f_c$ is a homogeneous sequence in $\fn^2$.  The local ring $R_{\mathfrak n}$ contains its residue field $k$.  The finite length modules 
$\Ext^{n}_{R_{\mathfrak n}}(M_{\mathfrak n},N_{\mathfrak n})$ are finite dimensional $k$-vector spaces; the same is true of the modules $\Ext^{n}_R(M,N)$, and $\dim_k\Ext^{n}_R(M,N)=\dim_k\Ext^{n}_{R_{\mathfrak n}}(M_{\mathfrak n},N_{\mathfrak n})$. Now the local result applies to show the proposition is true in the graded case.
\end{proof}

The next example gives modules with $\cx_R(M,N)=\min\{\dim M, \dim N\}$ and $h_R(M,N)\ne 0$. This shows that the additional conditions in part (3) of Theorem \ref{dimensions} are necessary. 

\begin{example}
Consider the ring $R=k[x_1, \dots, x_{2d}]/(x_1x_{d+1}, \dots, x_dx_{2d})$, for a field $k$, and the $R$-modules $M=R/(\ov x_1, \dots, \ov x_r)$, $N=R/(\ov x_{r+1}, \dots, \ov{x_{2d}})$, where $1\le r\le d$. We have then: 
\begin{numlist}
\item $\dim M=d=\dim R$  and $\dim N=r$;
\item $\cx_R(M,N)=r$;
\item $h_R(M,N)\ne 0$;
\item $R$ is regular in codimension $0$, but it does not have a unique minimal prime, and hence it is not regular in codimension $1$. 
\end{numlist}
\end{example}

\begin{proof} That $\dim R=d$ follows since $\dim k[x_1,\dots,x_{2d}]=2d$, and $x_1x_{d+1},\dots,x_dx_{2d}$ is a regular sequence of length $d$. Now one has that the module
$$M\cong k[x_{r+1},\dots,x_{2d}]/(x_{r+1}x_{r+1+d},\dots,x_dx_{2d})$$ has dimension $d$ since $k[x_{r+1},\dots,x_{2d}]$ has dimension $2d-r$, and the sequence $x_{r+1}x_{r+1+d},\dots,x_dx_{2d}$ is a regular sequence of length $d-r$. Also, the dimension of $N\cong k[x_1,\dots,x_r]$ is evidently $r$.

For $0\le j\le r$, define $N_j=N/(x_{j+1},\dots,x_r)N$.  For $1\leq j\leq r$, there are exact sequences
\[\xymatrix{
0\ar[r] & N_j \ar[r]^{x_j} & N_j \ar[r] & N_{j-1} \ar[r] & 0\,.
}\]
Since $x_j$ annihilates $\Ext^i_R(M,N_j)$ for all $i$, the resulting long exact sequence in cohomology
breaks into short exact sequences
\[
\xymatrix{
0\ar[r] & \Ext^i_R(M,N_j)\ar[r] & \Ext_R^i(M,N_{j-1})\ar[r] & \Ext^{i+1}_R(M,N_j)\ar[r] & 0\,.
}\]
Set $E_j(t)=\sum_{i=0}^\infty \dim_k\Ext^i_R(M,N_j)t^i$.  One has $E_{j-1}(t)=E_j(t)+\frac{1}{t}E_j(t)$.
Since $N_0\cong k$, we have $E_0(t)=\Po^R_M(t)$, and so we need to compute this Poincar\'e series.  This can be done along the lines of \cite[9.3.7]{Av}, but we include a different approach here.

The module $M_1=k[x_1,x_{d+1}]/(x_1)$ has a very simple free resolution $F_1$ over the ring $R_1=k[x_1,x_{d+1}]/(x_1x_{d+1})$; it is a periodic resolution comprised of rank one free modules where the maps alternate between multiplication by $x_1$ and multiplication by $x_{d+1}$. The same is true of the free resolutions $F_n$ of $M_n=k[x_n,x_{d+n}]/(x_n)$
over $R_n=k[x_n,x_{d+n}]/(x_nx_{d+n})$, for $1<n\le r$. Now we observe that 
\[
\textstyle{F_1\otimes_k\cdots\otimes_k F_r\otimes_k \frac{k[x_{r+1},\dots,x_d,x_{d+r+1},\dots,x_{2d}]}{(x_{r+1}x_{d+r+1},\dots,x_dx_{2d})}}
\]
is a free resolution of 
\[
\textstyle{M=M_1\otimes_k\cdots\otimes_k M_r\otimes_k \frac{k[x_{r+1},\dots,x_d,x_{d+r+1},\dots,x_{2d}]}{(x_{r+1}x_{d+r+1},\dots,x_dx_{2d})}}
\] 
over 
\[
\textstyle{R= R_1\otimes_k\cdots\otimes_k R_r\otimes_k \frac{k[x_{r+1},\dots,x_d,x_{d+r+1},\dots,x_{2d}]}{(x_{r+1}x_{d+r+1},\dots,x_dx_{2d})}}\,.
\]
It follows that $\Po^R_M(t)=\prod_{n=1}^r\Po^{R_n}_{M_n}(t)=1/(1-t)^r$.  Therefore
\[
E_r(t)=\frac{t^rE_0(t)}{(1+t)^r}=\frac{t^r}{(1-t^2)^r}
\]
This shows that $\cx_R(M,N)=r$.

From the expression for $E_r(t)$ one checks directly that $\Ext_R^i(M,N)$ is zero for all large even $i$ and non-zero for all large odd $i$ if $r$ is odd, and vice versa if $r$ is even.  This shows that $h_R(M,N)\ne 0$.

Property (4) is an easy exercise, having noted that the minimal primes of $R$ are of the form
$(x_{i_1},\dots,x_{i_d})$ where $i_j\in\{j,d+j\}$.
\end{proof}

\section*{Acknowledgment}
\noindent
We thank the anonymous referee for their careful reading of the manuscript and suggestions for improvement.


\end{document}